\documentclass[reqno]{amsart}

\usepackage{amsmath}    
\usepackage{amsthm}     
\usepackage{amssymb}    
\usepackage{graphicx}

\theoremstyle{remark}

\newtheorem{theorem}{Theorem}[section]

\newtheorem{definition}[theorem]{Definition}
\newtheorem{proposition}[theorem]{Proposition}
\newtheorem{corollary}[theorem]{Corollary}
\newtheorem{remark}[theorem]{Remark}

\newcommand{\diff}{ {\rm{d}} }

\begin{document}

\title[A ``Mean Value'' Approach to Asymptotic Behavior]{Exponential Decay Results for Semilinear Parabolic PDE with $C^0$ Potentials: A ``Mean Value'' Approach}

\author{Joseph L. Shomberg}
\address{Providence College \\ Department of Mathematics and Computer Science \\ Providence, RI 02918 \\ USA}
\email{jshomber@providence.edu}



\date{}

\begin{abstract}
The asymptotic behavior of some semilinear parabolic PDEs is analyzed by means of a ``mean value'' property. 
This property allows us to determine, by means of appropriate {\em{a priori}} estimates, some exponential decay results for suitable global solutions. 
We also apply the method to investigate a well-known finite time blow-up result.
An application is given to a one-dimensional semilinear parabolic PDE with boundary degeneracy. 
Our results shed further light onto the problem of determining initial data for which the corresponding solution is guaranteed to exponentially decay to zero or blow-up in finite time.
\end{abstract}

\maketitle

\section{Introduction}

Let $\Omega$ be a bounded domain (open and connected) subset of $\mathbb{R}^N$, for some positive integer $N$. Assume the boundary $\Gamma:=\partial\Omega$ is sufficiently smooth.
For $x\in\Omega$ and $t>0$, we study the asymptotic behavior of solutions $u=u(x,t)$ to semilinear parabolic equations for the form,
\begin{equation}
\label{eqn}
\frac{\partial u}{\partial t} - \nu \Delta u + f(u)u = 0,
\end{equation}
$\nu>0$, with the Dirichlet boundary conditions,
\begin{equation}
\label{bndry}
u_{\mid\Gamma}(x,\cdot)=0,
\end{equation}
and given the initial state,
\begin{equation}
\label{ic}
u(\cdot,0) = u_0(\cdot).
\end{equation}
(Of course the results could be suitably adjusted to incorporate other boundary conditions such as Neumann, mixed Neumann/Dirichlet, periodic, or Robin.) We only assume $f$ is a $C^0$ function on $\mathbb{R}$. Note that Hadamard well-posedness for problem (\ref{eqn})-(\ref{ic}) is not known because such with such minimal assumptions on $f$, uniqueness of solutions is not guaranteed. Typically for equations such as (\ref{eqn}), it is assumed that $f\in C^1(\mathbb{R})$ satisfy $f'(s)\geq \ell$, for some $\ell>0$ (cf. e.g. \cite[p. 213]{Robinson01}). 
Additionally, we cannot assume that the solutions are instantaneously regularizing.

The goal of this article is the provide a better description to the criteria that surrounds, not the well-posedness of problem (\ref{eqn})-(\ref{ic}), but rather the long-term behavior of the solutions to problem (\ref{eqn})-(\ref{ic}). The asymptotic behavior of solutions to PDE is a rich subject whose development we will only briefly mention. The study of dissipative dynamical systems is motivated by defining a solution operator for a given PDE, possibly posed abstractly as an ODE in a suitable Banach space, where the first task often is to demonstrate, besides global well-posedness, the existence of an absorbing set in the phase space. 
After that, one may demonstrate the solutions hold certain properties, like asymptotically smoothing.
In many efforts, the culmination of the study peaks with the existence of a global attractor, the maximal invariant subset of the phase space that attracts all trajectories.
This attractor is typically defined as the omega-limit set of a bounded absorbing set, and consists of smooth solutions. 
Some PDE also admit finite dimensional attractors whose rate of attraction is exponential. 

The study of dissipative dynamical systems and the development of attractors has flourished since the seminal work of  \cite{Babin&Vishik92,Ladyzhenskaya91,Temam88}. Furthermore, largely due to the permanent importance of the Navier-Stokes equations, attractors for PDE {\em{without unique solutions}} were also developed in \cite{Ball00} and \cite{Melnik&Valero98}. Indeed, generalized semiflows were employed in \cite{Ball04} and \cite{Segatti06} (just to name two applications). 
So-called trajectory dynamical systems were developed in \cite{Chepyzhov&Vishik02}.
Also, in the context of supercritical wave equations, the notion of trajectory dynamical systems appears in \cite{Zelik04}.

We will analyze the behavior of solutions for the above class of PDE in rather different terms: for guaranteed exponential convergence to zero. We will find conditions on the nonlinear term $f$ that guarantees the corresponding solutions exponentially decay to zero; hence, rendering some global solutions. Each of our decay results holds for all initial data $u_0\in H^1_0(\Omega)$ and all $\nu>0$. 
The criteria we use for each result depends on $f$ through a property we call the ``mean value'' of $f$ through $u$; named after the Mean Value Theorem for Integrals. 
It is important to note that we assume a solution $u$ (in the sense defined below) already exists, at least local in time. 
Hence, the estimates that follow are {\em{a priori}}, but insure a strict qualitative behavior for all nontrivial solutions.
The method is used to investigate a well-known blow-up result for semilinear parabolic PDE. For certain initial data, we find a positive time, using the mean value of the very solution, at which existence of the solution is no longer guaranteed.
In addition, an application of our method is given. 
This concerns a semilinear parabolic PDE with boundary degeneracy recently studied by \cite{Wang13} (surely an extension of problem (\ref{eqn})-(\ref{ic}) described above). 

Notation: $u_t=\frac{\partial u}{\partial t}$ and $\Delta$ denotes the Laplace differential operator on $\Omega$ with domain $H^2(\Omega)\cap H^1_0(\Omega)$. The symbols $|u|_p$ and $\|u\|$ denote the norm of $u$ in, respectively, $L^p(\Omega)$ and $L^2(\Omega)$, and $\langle u,u \rangle = \|u\|^2$. 
The measure of $\Omega$ is denoted by $|\Omega|:=\int_\Omega \diff x$.
Throughout, $\lambda_1>0$ will denote the Poincar\'{e} constant; $\|u\|^2 \leq \frac{1}{\lambda_1} \|\nabla u\|^2$. We write $H^{-1}(\Omega)$ to denote the dual of the space $H^1(\Omega)$. 
Finally, we will commonly identify $u(x,t)\equiv u(t)(x)$; e.g., $u(t)\in H^1_0(\Omega)$, $t>0$, and in many instances we will abbreviate $u(x,t)$ by simply $u$.

\section{The {\em{a priori}} results}  \label{s:2}

The following is the usual notion of a (weak) solution to problem (\ref{eqn})-(\ref{ic}). 

\begin{definition}
\label{d:weak-solution} 
Let $0<T\leq+\infty$. The function $u$ satisfying 
\begin{eqnarray}
u &\in& L^{\infty }(0,T;L^2(\Omega)) \cap L^{2}(0,T;H^1_0(\Omega)), \label{defn-1} \\
\partial_t u &\in& L^2(0,T;H^{-1}(\Omega)), \label{defn-2} 
\end{eqnarray}
is said to be a {\sc{weak solution}} to problem (\ref{eqn})-(\ref{ic}) if, for all $\varphi\in H^1_0(\Omega)$, and for almost all $t\in [0,T]$, there holds, 
\begin{equation*}
\label{weak-solution-1}
\langle u_t,\varphi \rangle + \nu
\langle \nabla u,\nabla \varphi \rangle + \langle f(u)u,\varphi \rangle = 0.
\end{equation*}
In addition, 
\begin{equation*}
\label{weak-solution-3}
u(0)=u_{0}. 
\end{equation*}
The function $[0,T]\ni t\mapsto u(t)\in H^1_0(\Omega)$ is called a {\sc{global weak solution}} if it is a weak solution for every $T>0$.
\end{definition}

Here it is well known that (cf. e.g. \cite[Lemma II.3.2]{Temam88}),
\[
u\in C([0,T];L^2(\Omega)),
\]
and for almost all $t\in[0,T]$, 
\begin{equation}
\label{Temam}
\frac{\diff}{\diff t}\|u\|^2 = 2\langle u_t,u \rangle.
\end{equation}

\begin{proposition}
\label{MV-lemma}
Let $T>0$. Suppose $u\in C([0,T];L^2(\Omega))$ and $f\in C(\mathbb{R})$.
Then there is $\xi\in C([0,T];\Omega)$ in which 
\begin{equation}
\label{xi}
\int_\Omega f(u(x,t))u^2(x,t) \diff x = f(u(\xi(t),t))\int_\Omega u^{2}(x,t) \diff x.
\end{equation}
\end{proposition}

\begin{proof}
The existence of the function $\xi\in C([0,T];\Omega)$ in (\ref{xi}) follows from the Mean Value Theorem for Integrals after some straightforward generalizations. See Theorem \ref{trick}.
\end{proof}

The following theorem provides the general result. 
It concerns the behavior of $f$ on solutions that are restricted to the path $(\xi(t),t)\in\Omega\times[0,T]$, where the function $\xi\in C([0,T];\Omega)$ is due to Proposition \ref{MV-lemma}.
Recall that, when $u$ is a solution to problem (\ref{eqn})-(\ref{ic}), then $u(\xi(t),t)\in\mathbb{R}$ for each $t\in[0,T]$.

\begin{theorem}
\label{general-result}
Let $\nu>0$, $f\in C(\mathbb{R})$ and $u_0\in H^1_0(\Omega)$. Suppose $u=u(x,t)$ is a weak solution to problem (\ref{eqn})-(\ref{ic}) for $t\in[0,T]$, for some $T>0$. 
Then $u$ satisfies, for all $t\in[0,T]$,
\begin{equation}
\label{general-condition}
\|u(t)\| \leq \exp\left[-\int_0^t f(u(\xi(\tau),\tau))\diff\tau \right] \|u_0\|.
\end{equation}
\end{theorem}

\begin{proof}
Let $u$ be a solution to problem (\ref{eqn})-(\ref{ic}) according to Definition \ref{d:weak-solution}. 
Since we are only interested in estimates at the {\em{a priori}} level, we are allowed to formally multiply (\ref{eqn}) by $u=u(x,t)$ in $L^2(\Omega)$ to obtain the identity, which holds for almost all $t\in(0,T)$,
\begin{equation}
\label{id-1}
\langle u_t,u \rangle - \nu \langle \Delta u,u \rangle + \langle f(u),u^2 \rangle = 0.
\end{equation}
First we recall (\ref{Temam}), then Green's first identity to see that there holds,
\[
\nu\langle -\Delta u, u \rangle = \nu\|\nabla u\|^2.
\]
Next, by Proposition \ref{MV-lemma}, we now know that there is $\xi\in C([0,T],\Omega)$ in which,
\begin{equation}
\label{id-2-f}
\langle f(u)u,u \rangle = f(u(\xi(t),t)) \|u\|^2.
\end{equation}
Together, (\ref{id-1}) becomes, for almost all $t\in(0,T)$, 
\begin{equation}
\label{id-2}
\frac{1}{2}\frac{\diff}{\diff t}\|u\|^2 + \nu\|\nabla u\|^2 + f(u(\xi(t),t)) \|u\|^2 = 0.
\end{equation}
Omitting the term $\nu\|\nabla u\|^2$ produces the differential inequality,
\begin{equation*}
\frac{1}{2}\frac{\diff}{\diff t}\|u\|^2 + f(u(\xi(t),t)) \|u\|^2 \leq 0,
\end{equation*}
and from this we find (\ref{general-condition}). 
\end{proof}

\begin{remark}
After multiplying identity (\ref{id-2}) by $\exp[\int_0^t 2f(u(\xi(\tau),\tau))\diff\tau]$ and integrating with respect to $t$ on $[0,T]$, we arrive at the new identity,
\[\begin{aligned}
& \|u\|^2 + 2\nu\int_0^t \exp\left[-\int_s^t 2f(u(\xi(\tau),\tau))\diff\tau \right]\|\nabla u(x,s)\|^2 \diff s \\
& = \exp\left[-\int_0^t 2f(u(\xi(\tau),\tau))\diff\tau \right]\|u_0\|^2.
\end{aligned}\]
Notice that for any $f\in C(\mathbb{R})$, we recover the bounds 
\[
u\in L^\infty(0,T;L^2(\Omega))\cap L^2(0,T;H^1_0(\Omega)),
\]
with dependence on $T>0$.
\end{remark}

With Theorem \ref{general-result}, we may now move onto the consideration of the case when solutions are guaranteed to exponentially decay to zero. Obviously, one immeditely read from (\ref{general-condition}) that exponential decay for solutions $u$ in $L^2(\Omega)$ occurs on the time intervals where
\[
F(t):=\int_0^t f(u(\xi(\tau),\tau)) \diff\tau < 0.
\]
However, given an arbitrary continuous function $f$, we seek conditions, with pragmatic assumptions on $f$, which guarantee solutions $u$ decay to zero exponentially, in $L^2(\Omega)$, as $t\rightarrow+\infty$.
We encounter the first assumption we can make that insures solutions to problem (\ref{eqn})-(\ref{ic}) decay exponentially to zero; when $f$ is {\em{positive}}.

\begin{corollary}
\label{exp-zero-result}
Let $f\in C(\mathbb{R})$ be a positive function on $\mathbb{R}$. Then for any $u_0\in H^1_0(\Omega)$ and $\nu>0$, the corresponding solutions $u$ to problem (\ref{eqn})-(\ref{ic}) are global ones (i.e., $T=+\infty$) and $\|u(t)\|$ exponentially decays to zero as $t\rightarrow +\infty$.
\end{corollary}

Recall that, for all $\varphi\in H^1_0(\Omega)$,
\begin{equation}
\label{Poincare}
\|\varphi\| \leq \frac{1}{\sqrt{\lambda_1}}\|\varphi_x\|,
\end{equation}
where $\lambda_1$ is the fist eigenvalue of the Laplacian with respect to homogeneous Dirichlet boundary conditions. Of course, (\ref{Poincare}) is known as the Poincar\'{e} inequality.

Under certain conditions, solutions will exponentially decay to zero when $f$ is not necessarily positive on all of $\mathbb{R}$. Because of the Poincar\'{e} inequality, $f$ may be allowed to assume some negative values, and we may still guarantee that solutions exponentially decay to zero.
The following extensions are for when $f\in C(\mathbb{R})$ is (eventually) {\em{bounded below}} by $-\nu\lambda_1$.

\begin{theorem}
\label{bounded}
Let $\nu>0$ and $f\in C(\mathbb{R})$.
Suppose $f$ satisfies the lower-bound,
\begin{equation}
\label{below}
\inf_{s\in\mathbb{R}} f(s) > -\nu\lambda_1.
\end{equation}
For any $u_0\in H^1_0(\Omega)$, the corresponding solutions $u$ to problem (\ref{eqn})-(\ref{ic}) are global ones (i.e., $T=+\infty$) and $\|u(t)\|$ exponentially decays to zero as $t\rightarrow +\infty$.
\end{theorem}

\begin{proof}
Let $u_0\in H^1_0(\Omega)$ and assume $u$ is a solution to problem (\ref{eqn})-(\ref{ic}) according to Definition \ref{d:weak-solution}. 
After applying the Poincar\'{e} inequality (\ref{Poincare}) to the identity (\ref{id-2}), we arrive at the differential inequality, which holds, for almost all $t\in(0,T)$,
\begin{equation}
\label{bnd-ineq-1}
\frac{\diff}{\diff t}\|u\|^2 + 2\left( \nu\lambda_1 + f(u(\xi(t),t)) \right)\|u\|^2 \leq 0.
\end{equation}
Integrating (\ref{bnd-ineq-1}) with respect to $t$ on $[0,T]$ yields,
\begin{equation}
\label{bnd-ineq-2}
\|u(t)\| \leq \exp \left[ - \int_0^t \left( \nu\lambda_1 + f(u(\xi(\tau),\tau)) \right) \diff \tau \right] \|u_0\|.
\end{equation}
At this point we recall assumption (\ref{below}) which together implies,
\begin{equation}
\label{bound-2}
f(u(\xi(t),t)) > -\nu\lambda_1, \quad\forall t\in[0,T].
\end{equation}
Thus, together (\ref{bnd-ineq-2}) and (\ref{bound-2}) show that $u$ is globally defined and exponentially decays to zero as $t\rightarrow +\infty$.
This completes the proof.
\end{proof}

Of course when $f$ is strictly bounded below by $-\nu\lambda_1$, we are allowed the take any initial data.
For the results that follow, solutions $u$ are assumed to be {\em{positive}}. Similar results for negative solutions can be show with minor modifications.
We will now show that for any initial data, positive solutions $u$ converge to zero exponentially (in $L^2(\Omega)$) when $f(s)$, $s\geq 0$, is bounded below in an appropriate manor.

\begin{theorem}
\label{eventual}
Let $\nu>0$ and $f\in C([0,+\infty))$. Suppose that the average value of $f(\sigma)$ on $(0,s)$ satisfies the lower-bound, for all $s>0$,
\begin{equation}
\label{mvp}
f_{\rm{avg}}(s) = \frac{1}{s} \int_0^s f(\sigma)\diff\sigma > -\nu\lambda_1.
\end{equation}
Then for any $u_0\in H^1_0(\Omega)$, the corresponding solutions $u$ to problem (\ref{eqn})-(\ref{ic}) are global ones (i.e., $T=+\infty$) and $\|u(t)\|$ exponentially decays to zero as $t\rightarrow +\infty$.
\end{theorem}

\begin{proof}
Let $u_0\in H^1_0(\Omega)$ and assume $u$ is a positive solution to problem (\ref{eqn})-(\ref{ic}) according to Definition \ref{d:weak-solution}. 
The claim follows directly by applying the assumption (\ref{mvp}) to (\ref{bnd-ineq-2}). Indeed, it suffices to show that, 
\[
\int_0^t \left( \nu\lambda_1 + f(u(\xi(\tau),\tau)) \right) \diff \tau >0, \quad\forall t\geq 0,
\]
or 
\[
-\nu\lambda_1 t < \int_0^t f(u(\xi(\tau),\tau))\diff\tau.
\]
Thanks to assumption (\ref{mvp}), the proof is complete.
\end{proof}

\subsection*{Examples}

1. First, we note that the well-known problem, the Chafee--Infante reaction diffusion equation,
\[
u_t - \nu\Delta u + u^3 - (\nu\lambda_1) u = 0,
\] 
satisfies condition (\ref{below}) because $f_1(s)=s^2-\nu\lambda_1$. (We will consider the finite time blow-up problem motivated by the nonlinear term $-f_1(s)s$ below.)

2. We also give an example of a function satisfying the condition (\ref{mvp}). Let us take $\nu=1$ and $\Omega=(0,1)$ so that $\lambda_1=1$ (cf. e.g. \cite{Hunter&Nachtergaele01} or \cite{Robinson01}). Define
\[
f_2(s) = \left\{ \begin{array}{ll}
1-s & \text{for $0 \leq s < 3$} \\ 
\frac{1-s}{s-2} & \text{for $3 \leq s$.}
\end{array} \right.
\]
Clearly, $f_2\in C^0([0,+\infty))$, and 
\[
f_{2\rm{avg}}(s):=\frac{1}{s}\int_0^s f_2(\sigma)\diff\sigma = \left\{ \begin{array}{ll}
1-\frac{s}{2} & \text{for $0\leq s <3$} \\ 
-1 -\frac{1}{s}\ln|s-2| +\frac{1}{s}\ln 2 & \text{for $3 \leq s$.}
\end{array} \right.
\]
See Figure \ref{fig3} below for the graphs of $f_2(s)$ and $f_{2\rm{avg}}(s)$. 

Some important properties of this function are:
\begin{enumerate}
\item[(i)] Here there holds,
\[
\liminf_{s\rightarrow+\infty}f(s) = -\nu\lambda_1,
\]
and notably $f_2(s)$ converges to $-\nu\lambda=-1$ from {\em{below}}; also  

\item[(ii)] the minimum average value of $f_2(s)$, on $(0,+\infty)$, is $-1$ and it is reached asymptotically, when $s\rightarrow+\infty$.
\end{enumerate}

\begin{figure}[htbp]
\centering
\includegraphics[scale=0.6]{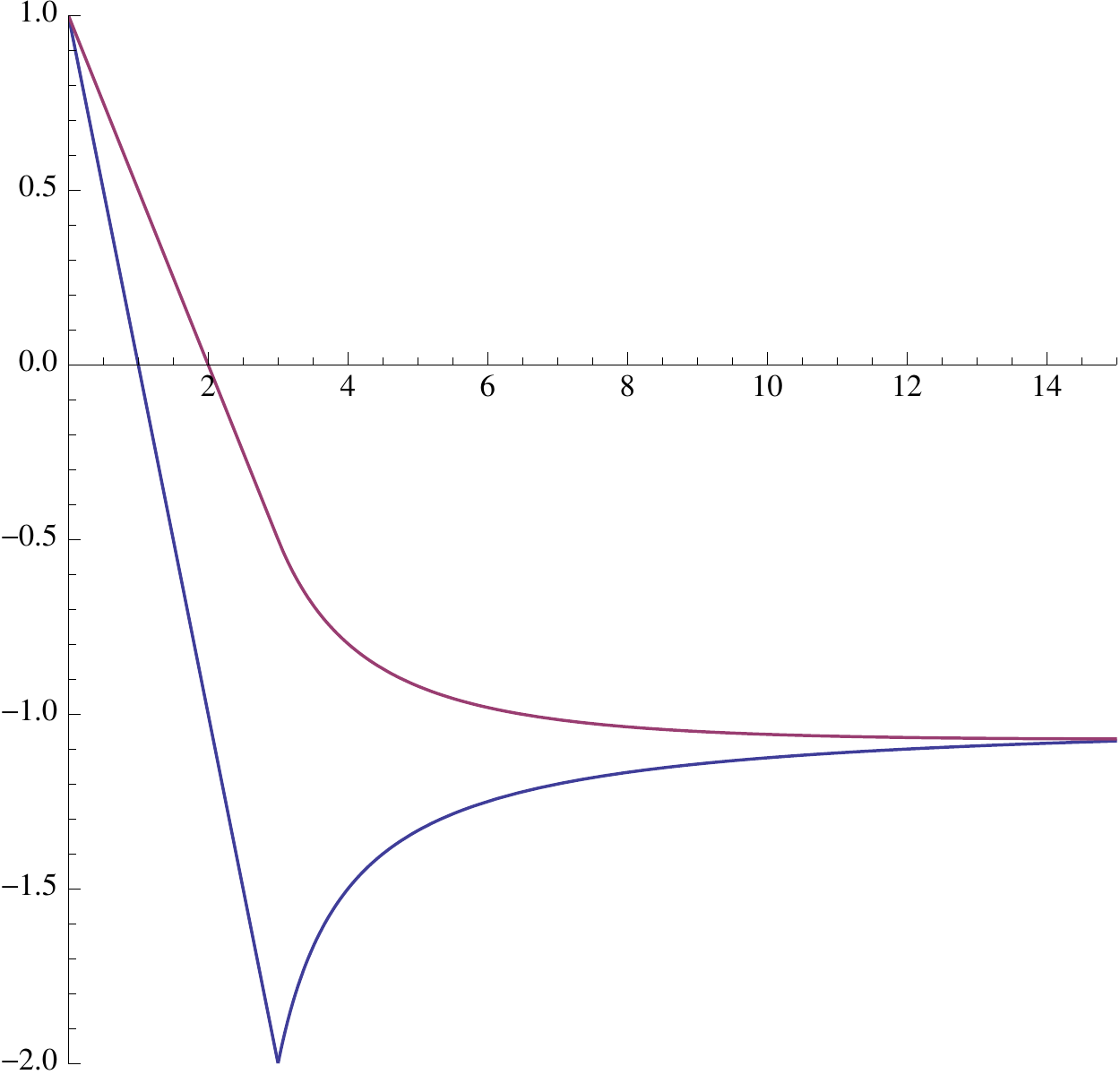}
\caption{The plot of $f_2(s)$ (blue) and $f_{2\rm{avg}}(s)$ (magenta) for $s\in(0,15]$.}
\label{fig3}
\end{figure}

Finally, we give an example of a nonlinear term $f$ satisfying the assumptions of Theorem \ref{eventual} where the nonlinear term $f$ is unbounded below. Consider the singular potential $f_3(s) = -(1-s)^{-p}$ on $s\in [0,1)$ for any $0 < p < 1 - \frac{1}{\nu\lambda_1}$. Then $f_3(s)$ is strictly decreasing on $[0,1)$, and 
\[
f_{3{\rm{avg}}}(s) = \int_0^1 -(1-s)^{-p} \diff s = \frac{1}{p-1} > -\nu\lambda_1.
\]

\begin{remark}
Notice that each assumption (\ref{below}) and (\ref{mvp}) implies that the same necessary condition on $f$ and the data $u_0$ holds:
\[
f(u(\xi(0),0)) = \frac{\langle f(u_0),u_0^2 \rangle}{\|u_0\|^2} > -\nu\lambda_1.
\]
\end{remark}

The final result in this section is motivated by the blow-up result in \cite[p. 176]{Zheng04}. Indeed, we give a description of a blow-up condition in terms of the ``mean value'' technique developed thus far. 

\begin{theorem}  \label{blow-up-result}
Let $\nu>0$ and $f\in C((0,+\infty))$. Suppose $f$ satisfies the upper- and lower-bounds, for all $s\in (0,+\infty)$,
\begin{equation}  \label{blow-up}
c_1 - \frac{c_2(r+2)}{4} s^r \leq f(s) \leq c_1 - c_2 s^{r},
\end{equation} 
for some $c_1\geq 0$, $c_2>0$ and some $r > 2$.
If $u_0\in H^1_0(\Omega)$ ($u_0>0$) satisfies the condition, 
\begin{equation}
\label{criteria-blow-up}
\nu\|\nabla u_0\|^2 + c_1\|u_0\|^2 < \frac{c_2}{2} |u_0|^{r+2}_{r+2},
\end{equation}
then the corresponding positive solutions $u$ to problem (\ref{eqn})-(\ref{ic}) possess the interval of existence $(0,t')$, where, 
\begin{equation}
\label{time}
t' = \frac{2|\Omega|^{(r-2)/2}}{c_2(r-2) u^2(\xi(t^*),t^*) \|u_0\|^{r-2}},
\end{equation}
for some $t^*=t^*(t)>0$.
\end{theorem}

\begin{proof}
Let $u_0\in H^1_0(\Omega)$ satisfy (\ref{criteria-blow-up}) and assume $u$ is a solution to problem (\ref{eqn})-(\ref{ic}) according to Definition \ref{d:weak-solution}. 
This time we apply (\ref{blow-up}) to (\ref{id-2-f}) to find that there holds, 
\[\begin{aligned}
\langle f(u)u,u \rangle & = \int_\Omega f(u)u^2 \diff x \\ 
& \leq \int_\Omega \left( c_1 - c_2 u^r \right) u^2 \diff x \\ 
& \leq \int_\Omega c_1 u^2 \diff x - \frac{1}{2} \int_\Omega c_2 u^2 u^{r} \diff x - \frac{1}{2} \int_\Omega c_2 u^2 u^{r} \diff x \\ 
& = c_1 \|u\|^2 - \frac{c_2}{2} |u|^{r+2}_{r+2} - \frac{c_2}{2} u^2(\xi(t),t) |u|^r_r.
\end{aligned}\]
With this, instead of (\ref{id-2}), we find that there holds, for almost all $t\in(0,T)$,
\begin{equation}  \label{blow-up-2}
\frac{1}{2}\frac{\diff}{\diff t}\|u\|^2 + \nu\|\nabla u\|^2 + c_1 \|u\|^2 - \frac{c_2}{2} |u|^{r+2}_{r+2} \geq \frac{c_2}{2} u^2(\xi(t),t) |u|^r_r.
\end{equation}
Define the functional $\mathcal{E}:H^1_0(\Omega)\rightarrow\mathbb{R}$, for all $\varphi\in H^1_0(\Omega)\cap L^{r+2}(\Omega)$, by
\[
\mathcal{E}(\varphi) := \nu\|\nabla \varphi\|^2 + c_1\|\varphi\|^2 - \frac{c_2}{2} |\varphi|^{r+2}_{r+2}.
\]
Let $\mathcal{E}(t)$ denote the functional $\mathcal{E}$ along trajectories $\varphi=u(t)$.
We observe that 
\[\begin{aligned}
\frac{1}{2}\frac{\diff\mathcal{E}(t)}{\diff t} & = \langle -\nu\Delta u + \left( c_1 - \frac{c_2}{4}(r+2) u^r \right)u, u_t \rangle \\ 
& \leq -\|u_t\|^2.
\end{aligned}\]
Thus, for all $t\in(0,T)$,
\[
\mathcal{E}(0) \geq \mathcal{E}(t),
\]
and since (\ref{criteria-blow-up}) holds, $\mathcal{E}(0)<0$.
With this, (\ref{blow-up-2}) becomes 
\[
\frac{\diff}{\diff t}\|u\|^2 \geq c_2 u^2(\xi(t),t) |u|^r_r,
\]
whereby with H\"{o}lder's inequality,
\[
|u|^r_r \geq |\Omega|^{(2-r)/2} \|u\|^r,
\]
so we now have, for almost all $t\in(0,T)$, 
\begin{equation}  \label{blow-up-3}
\frac{\diff}{\diff t}\|u\|^2 \geq c_2 |\Omega|^{(2-r)/2} u^2(\xi(t),t) \|u\|^r.
\end{equation}
Integrating (\ref{blow-up-3}) with respect to $t$ over $[0,T]$ yields,
\begin{equation}
\label{time-2}
\|u(t)\|^2 \geq \left( \|u_0\|^{2-r} + \frac{c_2(2-r)}{2} |\Omega|^{(2-r)/2} \int_0^t u^2(\xi(\tau),\tau) \diff\tau \right)^{2/(2-r)}.
\end{equation}
Recall that $r>2$; whence, finite time blow-up occurs when 
\[
\int_0^t u^2(\xi(\tau),\tau) \diff\tau = \frac{2|\Omega|^{(r-2)/2}}{c_2(r-2) \|u_0\|^{r-2}}.
\]
We now appeal to the Mean Value Theorem for Integrals once again; there is $t^*\in (0,t)$ in which 
\[
\int_0^t u^2(\xi(\tau),\tau) \diff\tau = u^2(\xi(t^*),t^*)t.
\]
Thus, the right-hand side of (\ref{time-2}) is singular whenever (note, $t$ is positive),
\[
t = \frac{2|\Omega|^{(r-2)/2}}{c_2(r-2) u^2(\xi(t^*),t^*) \|u_0\|^{r-2}}.
\]
This shows (\ref{time}) as claimed.
\end{proof}

\begin{remark}
The preceding result guarantees the local existence of a solution to problem \eqref{eqn}-\eqref{ic}.
In the proof, $t^*$ depends on $t$, so we cannot claim that we know for which $t$ a finite time blow-up occurs.
Nevertheless, this shows that the ``MVT method'' can be used to show that certain problems cannot possess global solutions.
Below, we will see how the method can be used to show finite time blow-up.
\end{remark}

\section{Application}

In this section we encounter a type of perturbation for the problem considered in \cite{Fujita-66}. The study of the equation,
\[
u_t -\Delta u - u^{\alpha+1} = 0, \quad\alpha>0,
\]
has led to the development of the study of blow-up solutions of PDE.
We will consider, for $x\in(0,1)$ and $t>0$, the asymptotic behavior of positive solutions $u=u(x,t)$ of the semilinear parabolic equation with boundary degeneracy,
\begin{equation}
\label{ap-eqn}
u_t = \left( x^d u_x \right)_x + u^p,
\end{equation}
$d>0$, $p>1$, with the (mixed weighted Neumann and Dirichlet) boundary conditions,
\begin{equation}
\label{ap-bndry}
x^d u_x(0,\cdot)=0, \quad u(1,\cdot)=0,
\end{equation}
and
\begin{equation}
\label{ap-ic}
u(\cdot,0) = u_0(\cdot).
\end{equation}

Of particular interest is the recent work by \cite{Wang13}, which we now report.
First a definition (cf. \cite[Definition 2.1]{Wang13}).

\begin{definition}
\label{def-Wang}
Assume that $0<T\leq +\infty$. A nonnegative function $u$ is called a {\sc{solution}} to problem (\ref{ap-eqn})-(\ref{ap-ic}) if 
\begin{enumerate}
\item[(i)] For any $0<\tilde T<T$, $u\in L^\infty((0,1)\times(0,\tilde T))$ with $u_t\in L^2((0,1)\times(0,\tilde T))$ and $x^{d/2}u_x\in L^2((0,1)\times(0,\tilde T))$.

\item[(ii)] For any $0<\tilde T<T$ and any nonnegative $\varphi\in C^1([0,1]\times[0,\tilde T])$ vanishing at $x=1$,
\[\begin{aligned}
\int_0^{\tilde T} \int_0^1 & \left( u_t(x,t)\varphi(x,t) + x^d u_x(x,t)\varphi(x,t) \right) \diff x \diff t \\
& = \int_0^{\tilde T} \int_0^1 u^p(x,t)\varphi(x,t)\diff x \diff t.
\end{aligned}\]

\item[(iii)] $u(1,\cdot) = 0$ on $(0,T)$ and $u(\cdot,0) = u_0(\cdot)$ on $(0,1)$ (in the sense of trace).
\end{enumerate}
\end{definition}

The results we investigate are (cf. \cite[Theorems 2.1, 2.2]{Wang13}).

\begin{theorem}  \label{thm}
Assume $0< d <2$. Then there exists both nontrivial global and blowing-up solutions to problem (\ref{ap-eqn})-(\ref{ap-ic}).
\end{theorem}

\begin{theorem}  \label{blow-up-thm}
Assume $d \ge 2$.
The solution to problem (\ref{ap-eqn})-(\ref{ap-ic}) must blow up in a finite time for any nontrivial $0 \le u_0 \in L^\infty(0,1)$ with $x^{d/2}\partial_x u_0 \in L^2(0,1)$.
\end{theorem}

The purpose of this application is to provide a better description to the states $u_0$ for which the corresponding solution converges exponentially to zero, or leads to finite time blow-up. 
This is done in terms of the structural parameters $d$ and $p$. 
Our idea employs a similar argument that produces the results in the previous section; i.e., we find an {\em{a priori}} estimate which exploits a generalization of the Mean Value Theorem for Integrals to derive a condition that guarantees solutions $u$ (to a comparable problem) exponentially converge to zero.
Because of the nature of the equation examined here, we are able to explicitly determine a finite time blow-up result as well.

It was crucial to the development of the results in section \ref{s:2} that we explicitly know the value of the first eigenvalue of the Laplacian with respect the the imposed boundary conditions, we now introduce the comparable problem which we will use to illustrate Theorem \ref{thm}.
Observe that any solution to \eqref{ap-eqn} satisfying the homogeneous mixed Neumann--Dirichlet boundary condition
\begin{equation}
\label{ap-bndry2}
u_x(0,\cdot)=0, \quad \text{and} \quad u(1,\cdot)=0,
\end{equation}
also satisfies \eqref{ap-bndry} when $d>0$.
Then recall that the Laplacian on $(0,1)$ subject to the Neumann--Dirichlet boundary conditions \eqref{ap-bndry2} admits the following system of eigenvalues and eigenvectors, \footnote{https://en.wikipedia.org/wiki/Eigenvalues\_and\_eigenvectors\_of\_the\_second\_derivative\#Mixed\_Neumann-Dirichlet\_boundary\_conditions}
\[
\lambda_k=\frac{(2k-1)^2\pi^2}{4}, \quad \phi_k(x)=\sqrt{2}\cos\left( \frac{(2k-1)\pi x}{2} \right), \quad k=1,2,3,\dots.
\]
Whence, on $(0,1)$, the Poincar\'{e} constant (first eigenvalue of the Laplacian subject to \eqref{ap-bndry2}) is $\lambda_1=\frac{\pi^2}{4}$. 

Our first result in response to Theorem \ref{thm} is the following

\begin{theorem}
\label{ap-result}
Let $T>0$, $0 < d < 2$ and $p>1$. Suppose $u=u(x,t)$ is a positive solution to problem (\ref{ap-eqn}), (\ref{ap-ic}), and \eqref{ap-bndry2}. 
There are $\xi,\chi\in C([0,T];(0,1))$ in which 
\begin{equation}  \label{ap-xi}
\int_0^1 x^d u_x^2(x,t) \diff x = \xi^d(t) \int_0^1 u_x^2(x,t) \diff x,
\end{equation}
and
\begin{equation}
\label{ap-chi}
\int_0^1 u^{p+1}(x,t) \diff x = u^{p-1}(\chi(t),t)\int_0^1 u^{2}(x,t) \diff x.
\end{equation}
Furthermore, for any $u_0\in H^1(\Omega)$, $0<d<2$ and $p>1$ satisfying the condition,
\begin{equation}
\label{Wang-cond}
\frac{|u_0|^{p+1}_{p+1}}{\|u_0\|^2} < \frac{\pi^2}{4}\frac{\langle x^d,(\partial_x u_0)^2 \rangle}{\|\partial_x u_0\|^2},
\end{equation}
then the solution of problem (\ref{ap-eqn}), (\ref{ap-ic}), and \eqref{ap-bndry2} exponentially decays to zero for all $t\in(0,T)$.
\end{theorem}

\begin{proof}
Suppose $u$ is a solution to problem (\ref{ap-eqn}), (\ref{ap-ic}), and \eqref{ap-bndry2} in the sense of Definition \ref{def-Wang}. 
The existence of the functions $\xi,\chi\in C([0,T];(0,1))$ in (\ref{ap-xi}) and (\ref{ap-chi}) again follows from the Mean Value Theorem for Integrals (see Theorem \ref{trick}). It remains to show the condition which guarantees the solution's exponential decay to zero. 
As usual, we begin by multiplying (\ref{ap-eqn}) by $u=u(x,t)$ in $L^2(0,1)$ to obtain the identity,
\begin{equation}
\label{ap-id-1}
\frac{1}{2}\frac{\diff}{\diff t}\|u\|^2 = \left\langle \partial_x\left( x^d u_x \right), u \right\rangle + \left\langle u^{p+1},1 \right\rangle.
\end{equation}
Applying the Mean Value Theorem for Integrals, we obtain, 
\[\begin{aligned}
\left\langle \partial_x\left( x^d u_x \right), u \right\rangle & = \int_0^1 \partial_x\left( x^d u_x(x,t) \right) u \diff x = -\int_0^1 x^d u_x^2(x,t) \diff x \\ 
& = -\xi^d(t) \int_0^1 u_x^2(x,t) \diff x = -\xi^d(t) \|u_x\|^2,
\end{aligned}\]
for some $\xi(t)\in C([0,T],(0,1))$.  For some $\chi(t)\in C([0,T],(0,1))$, we also find
\[\begin{aligned}
\left\langle u^{p+1}, 1 \right\rangle & = \int_0^1 u^{p+1}(x,t) \diff x = \int_0^1 u^{p-1}(x,t) u^2(x,t) \diff x \\ 
& = u^{p-1}(\chi(t),t) \int_0^1 u^2(x,t) \diff x = u^{p-1}(\chi(t),t)\|u\|^2.
\end{aligned}\]
So (\ref{ap-id-1}) becomes 
\begin{equation}
\label{ap-id-2}
\frac{\diff}{\diff t}\|u\|^2 + 2\xi^d(t) \|u_x\|^2 = 2u^{p-1}(\chi(t),t)\|u\|^2.
\end{equation}

Recall $\lambda_1=\frac{\pi^2}{4}$. 
Hence, from (\ref{id-2}) we arrive at the differential inequality
\begin{equation}
\label{ap-ineq-1}
\frac{\diff}{\diff t}\|u\|^2 + 2\left( \frac{\pi^2}{4}\xi^d(t) - u^{p-1}(\chi(t),t) \right)\|u\|^2 \leq 0.
\end{equation}
Thus, integrating (\ref{ap-ineq-1}) with respect to $t$ on $[0,T]$ yields,
\begin{equation}
\label{ap-ineq-2}
\|u(t)\| \leq \exp \left[ - \int_0^t \left( \frac{\pi^2}{4}\xi^d(\tau)  - u^{p-1}(\chi(\tau),\tau) \right) \diff \tau \right] \|u_0\|,
\end{equation}
Hence, exponential decay to zero is guaranteed when, 
\[
u^{p-1}(\chi(t),t) < \frac{\pi^2}{4}\xi^d(t), \quad \forall t\in(0,T),
\]
and by monotonicity (in this case the decreasing function is $f(s)= -s^{p-1}$ for $s>0$ and $p>1$), 
\begin{equation}
\label{ap-cond}
u^{p-1}_0(\chi(0)) < \frac{\pi^2}{4}\xi^d(0).
\end{equation}
Thus, in terms of the initial data $u_0$, thanks to (\ref{ap-xi}) and (\ref{ap-chi}), the condition (\ref{ap-cond}) becomes (\ref{Wang-cond}) as claimed.
This finishes the proof.
\end{proof}

\begin{remark}
The conditions that guarantee finite time blow-up when $0<d<2$ are not as clear. However, there is the partial result. After applying the Cauchy-Schwarz and Poincar\'{e} inequalities to the condition (\ref{Wang-cond}) we see that, when 
\[
\frac{|\partial_x u_0|^2_4}{|u_0|^{p+1}_{p+1}} \leq \frac{\pi^2}{4},
\]
then blow-up in finite time is {\em{possible}} for any $0<d<2$ and $p>1$.
\end{remark}

We are now in position to provide a better description for the solutions that converge exponentially to zero. 
Using (\ref{Wang-cond}), we are motivated to define the set on which solutions to problem (\ref{ap-eqn}), (\ref{ap-ic}), and \eqref{ap-bndry2} are guaranteed to exponentially decay to zero,
\[
\mathcal{E} = \left\{ u_0\in H^1_0(\Omega) : u_0\not\equiv 0, \frac{\|u_0\|^2 |\partial_x u_0|^2_4}{|u_0|^{p+1}_{p+1} \|\partial_x u_0\|^2} > \frac{4}{\pi^2}\sqrt{1+2d} \right\}.
\]

It is possible to further illustrate this application with an inverse problem. Given $u_0(x)$, one computes $|\partial_x u_0|^2_4$, $\|u_0\|^2$, $|u_0|^{p+1}_{p+1}$, and $\|\partial_x u_0\|^2$. 
Since $0<d<2$, the set of $p>1$ which guarantee the corresponding solution $u$ decays exponentially to zero can be implicitly determined from the inequality
\[
\frac{\sqrt{1+2d} |u_0|^{p+1}_{p+1} \|\partial_x u_0\|^2} {\|u_0\|^2 |\partial_x u_0|^2_4} < \frac{4}{\pi^2} \approx 0.405285.
\]
Now the left hand side also defines a function of the two variables $(d,p)$. Hence, when we are given $u_0(x)$, the region containing the value of $(d,p)$ which guarantees the solution converges exponentially to zero can be found on the contour map of the surface described by that function. We illustrate this with three different initial conditions, $u_0(x)=x\sin(\pi x)$ and $u_0(x)=e^{1-x^2}-1$ (see Figures \ref{fig1} and \ref{fig2} below).

\begin{figure}[htbp]
\centering
\includegraphics[scale=0.5]{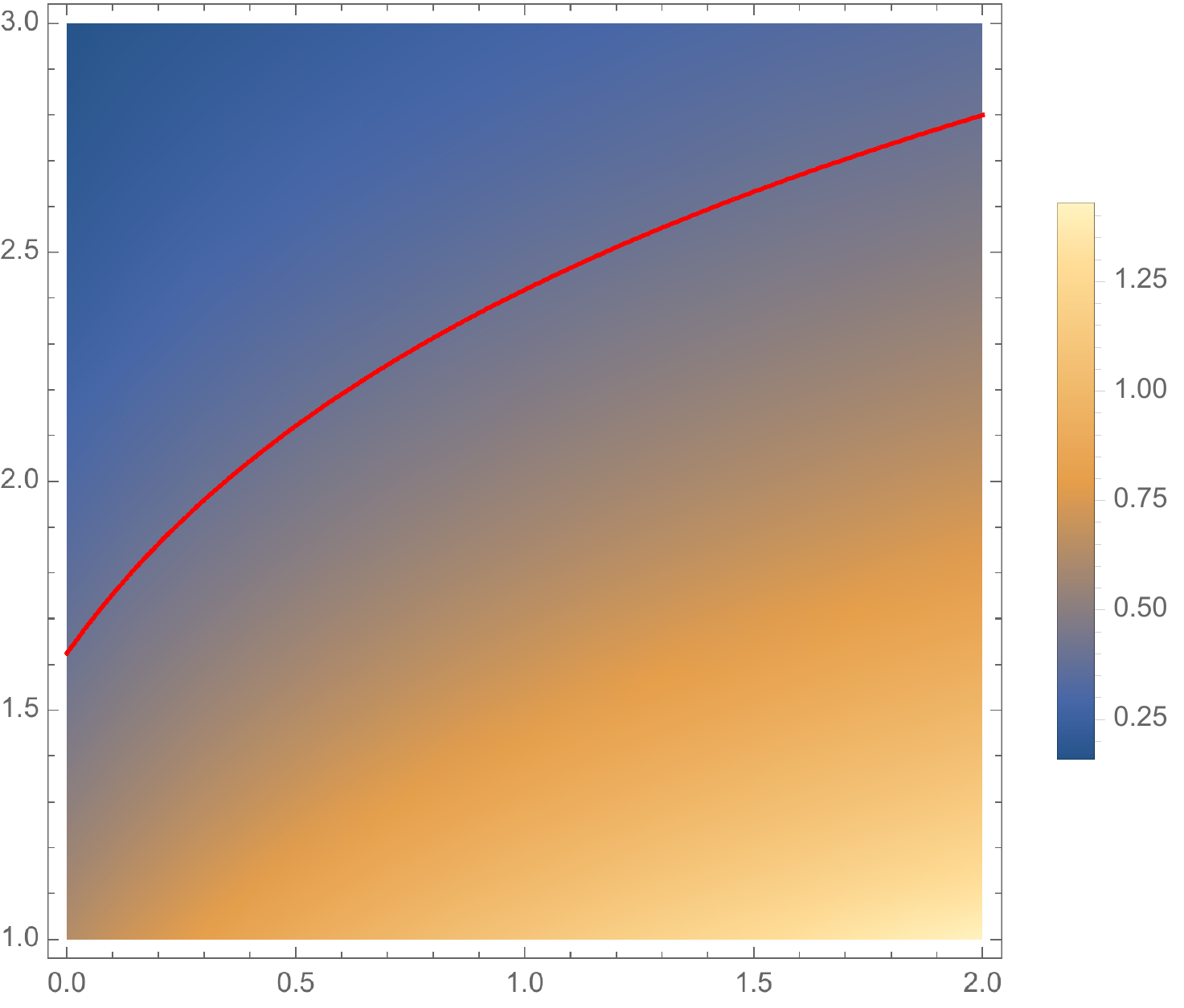}
\caption{Given $u_0(x)=x\sin(\pi x)$. Plot of the parameter space: $d\in[0,2]$, $p\in[1,3]$. Parameters $(d,p)$ above the critical contour (the $4/\pi^2$ level curve shown in red) guarantee the solution exponentially decays to zero.}
\label{fig1}
\end{figure}

\begin{figure}[htpb]
\centering
\includegraphics[scale=0.5]{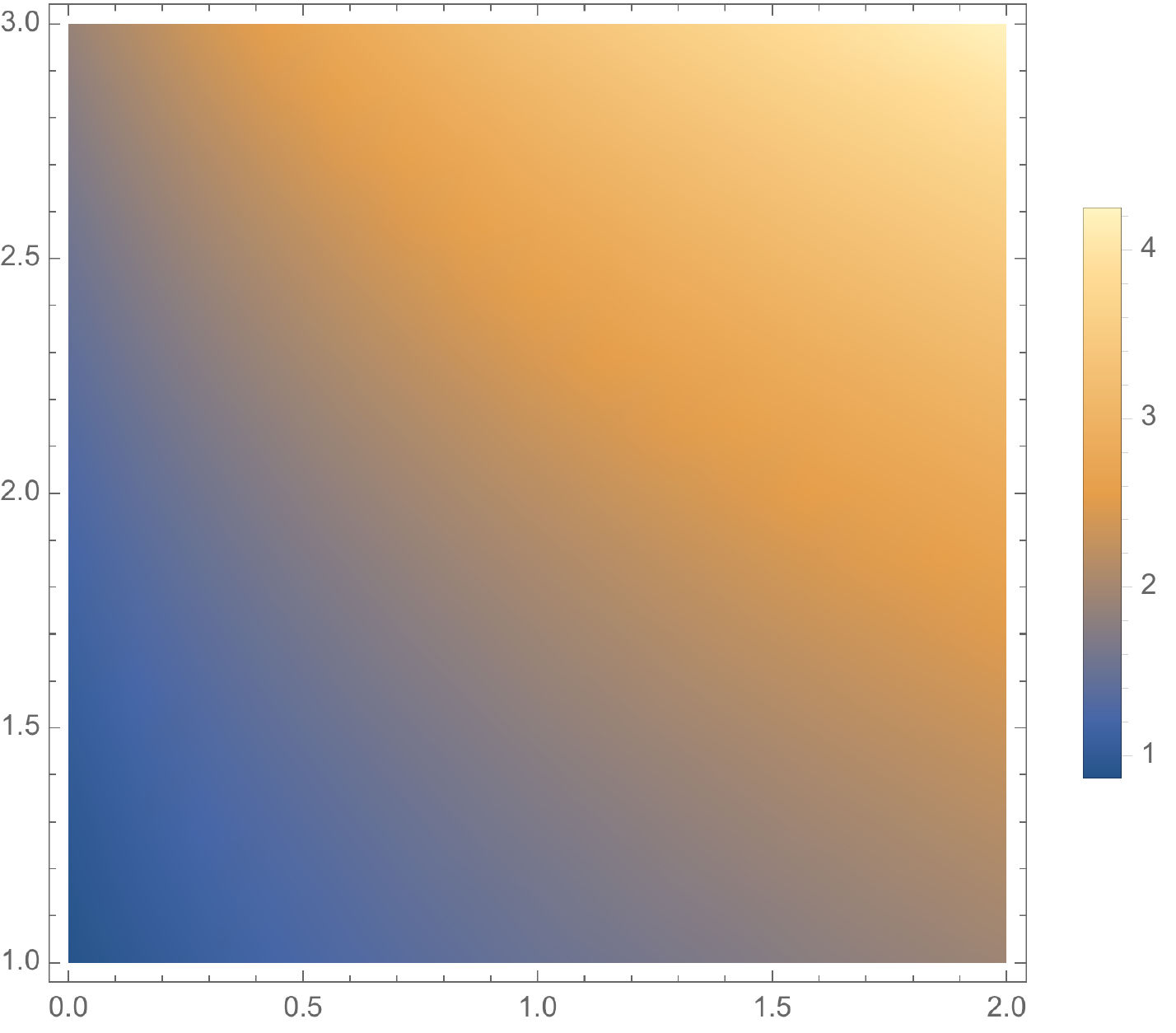}
\caption{Given $u_0(x)=e^{1-x^2}-1$. Plot of the parameter space: $d\in[0,2]$, $p\in[1,3]$. All parameters $(d,p)$ in this space guarantee the solution exponentially decays to zero.}
\label{fig2}
\end{figure}

Finally, our last result is in response to Theorem \ref{blow-up-thm}.
In this case we consider the original problem \eqref{ap-eqn}-\eqref{ap-ic} described in the application (no knowledge of the best Poincar\'{e} constant is needed here).

\begin{theorem}  \label{ap-result2}
Let $T>0$, $d>0$ and $p>1$. Suppose $u=u(x,t)$ is a positive solution to problem (\ref{ap-eqn})-(\ref{ap-ic}). 
There is $\xi\in C([0,T];(0,1))$ in which 
\begin{equation}
\label{ap-xi2}
\int_0^1 x^d u_x^2(x,t) \diff x = \xi^d(t) \int_0^1 u_x^2(x,t) \diff x,
\end{equation}
Furthermore, for any $u_0\in H^1(\Omega)$, $d>0$ and $p>1$ satisfying the condition,
\begin{equation}  \label{Wang-cond2}
\langle x^d,(\partial_x u_0)^2 \rangle < \frac{1}{p+1} |u_0|^{p+1}_{p+1},
\end{equation}
then the corresponding positive solutions $u$ to problem (\ref{ap-eqn})-(\ref{ap-ic}) blow-up at the finite time
\begin{equation}  \label{time2}
t' = \frac{p+1}{2p} \frac{1}{\|u_0\|^\frac{p-1}{2}},
\end{equation}
that is, $\lim_{t\rightarrow t'^-} \|u(t)\| = +\infty$.
\end{theorem}

\begin{proof}
The assertion in \eqref{ap-xi2} was already shown in the proof of the previous theorem.

We proceed as in the proof of Theorem \ref{blow-up-result}.
Multiplying \eqref{ap-eqn} by $u$ in $L^2(0,1)$ and applying \eqref{ap-xi} we obtain,
\begin{equation}  \label{ap-poi-0}
\frac{1}{2}\frac{\diff}{\diff t}\|u\|^2 + \langle -\partial_x(x^du_x),u \rangle - |u|^{p+1}_{p+1} = 0.
\end{equation}
A further multiplication of \eqref{ap-eqn} by $u_t$ produces,
\[
\frac{\diff}{\diff t} \left\{ \langle -\partial_x(x^du_x),u \rangle - \frac{1}{p+1}|u|^{p+1}_{p+1} \right\} + \|u_t\|^2 = 0.
\]
We define the functional 
\begin{equation}  \label{ap-poi-1}
E(t):=\langle -\partial_x(x^du_x(t)),u(t) \rangle - \frac{1}{p+1}|u(t)|^{p+1}_{p+1},
\end{equation}
and thanks to \eqref{ap-poi-1}, we immediately know that
\[
E(t)\le E(0).
\]
From \eqref{ap-poi-0} we also have,
\[
\frac{1}{2}\frac{\diff}{\diff t}\|u\|^2 + E - \frac{p}{p+1} |u|^{p+1}_{p+1} = 0,
\]
and if $E(0)<0$, then 
\[
\frac{1}{2}\frac{\diff}{\diff t}\|u\|^2 \ge \frac{p}{p+1} |u|^{p+1}_{p+1} \ge \frac{p}{p+1} \|u\|^{2\cdot\frac{p+1}{2}}.
\]
We now observe that, due to the complicated degenerate boundary conditions imposed on the solution in \eqref{ap-bndry}, we find following integration by parts and recalling \eqref{ap-xi2},  the condition \eqref{Wang-cond2}.
Indeed, consider
\begin{align*}
E(0) & = \langle -\partial_x(x^d \partial_x u_0),u_0 \rangle - \frac{1}{p+1}|u_0|^{p+1}_{p+1}  \\
& = \langle x^d \partial_x u_0,\partial_xu_0 \rangle - \frac{1}{p+1}|u_0|^{p+1}_{p+1}  \\
& = \xi^d(0)\|\partial_x u_0\|^2 - \frac{1}{p+1}|u_0|^{p+1}_{p+1}  \\
& = \langle x^d,(\partial_x u_0)^2 \rangle - \frac{1}{p+1}|u_0|^{p+1}_{p+1}.  \\
\end{align*}
Now following some of the calculations as in the proof of Theorem \ref{blow-up-result}, we readily find the blow-up time \eqref{time2} as claimed.
This finishes the proof.
\end{proof}

\section*{Acknowledgments}

The author would like to thank the referees for their generous comments when reading the manuscript.

\appendix
\section{}

Here we report the Mean Value Theorem for Integrals that is used throughout this paper.
The interested reader can also see \cite[Chapter 7]{SR}.

\begin{theorem}
\label{trick}
Let $T>0$ and $\Omega$ be a bounded domain (open and connected) in $\mathbb{R}^N$. 
Let $f,\varphi\in C([0,T];L^1(\Omega))$. Suppose $\varphi(t)$ is nonnegative a.e. on $\Omega$. 
Then there is $\xi(t)\in C([0,T];\Omega)$ in which, for all $t\in[0,T]$, 
\[
\int_\Omega f(x,t)\varphi(x,t) \diff x = f(\xi(t),t)\int_\Omega\varphi(x,t) \diff x,
\]
that is,
\[
\langle f(t),\varphi(t) \rangle_{L^2(\Omega)} = f(\xi(t),t) |\varphi(t)|_{L^1(\Omega)}.
\]
\end{theorem}

\begin{proof}
Denote by $C_c(\Omega)$ the continuous functions on $\Omega$ that are compactly supported in $\Omega$ (recall that such functions are dense in $L^p(\Omega)$, $1\leq p<\infty$).
Fix $t'\in[0,T]$ and let $\tilde f(t'),\tilde\varphi(t')\in C_c(\Omega)$, where $\tilde\varphi(t')$ is nonnegative on $\Omega$.
By the Extreme Value Theorem for continuous functions, there exist $x_m$ and $x_M$ in $\Omega$ so that 
\[
m(t'):=\min_{x\in\Omega}\left\{\tilde f(x,t')\right\} = \tilde f(x_m,t')
\]
and 
\[
M(t'):=\max_{x\in\Omega}\left\{\tilde f(x,t')\right\} = \tilde f(x_M,t').
\]
Hence, there holds,
\[
m(t') \int_\Omega \tilde\varphi(x,t') \diff x \leq \int_\Omega \tilde f(x,t')\tilde\varphi(x,t') \diff x \leq M(t') \int_\Omega \tilde\varphi(x,t') \diff x.
\]
Since $\tilde\varphi(t')$ is nonnegative, $|\tilde\varphi(t')|_{L^1(\Omega)}=\int_\Omega \tilde\varphi(x,t') \diff x$. In the case when $\tilde\varphi(t')=0$, we obtain equality and identity. So in the case when $\tilde\varphi(t')$ is nonzero, we obtain 
\[
m(t') \leq \frac{1}{|\tilde\varphi(t')|_{L^1(\Omega)}}\int_\Omega \tilde f(x,t')\tilde\varphi(x,t') \diff x \leq M(t').
\]
Since $\tilde f(x,t')$ is continuous and $\tilde f(x,t')\in[m(t'),M(t')]$, then the Intermediate Value Theorem provides an $\xi(t')\in\Omega$ in which 
\[
\tilde f(\xi(t'),t') = \frac{1}{|\tilde\varphi(t')|_{L^1(\Omega)}}\int_\Omega \tilde f(x,t')\tilde\varphi(x,t') \diff x.
\]
Thus, 
\[
\int_\Omega \tilde f(x,t')\tilde\varphi(x,t') \diff x = \tilde f(\xi(t'),t') \int_\Omega \tilde\varphi(x,t') \diff x.
\]
By the density of $C_c(\Omega)\subset L^1(\Omega)$, we also have for $f(t'),\varphi(t')\in L^1(\Omega)$, 
\[
\int_\Omega f(x,t')\varphi(x,t') \diff x = f(\xi(t'),t') \int_\Omega \varphi(x,t') \diff x,
\] 
or 
\begin{equation}
\label{mvt-1}
\langle f(t')\varphi(t') \rangle_{L^2(\Omega)} = f(\xi(t'),t') |\varphi(t')|_{L^1(\Omega)}.
\end{equation}
So far we have shown that, for each fixed $t'\in[0,T]$, there exists $\xi(t') \in \Omega$ in which (\ref{mvt-1}) holds. 
It remains to show that the map $[0,T]\ni t\mapsto \xi(t)\in\Omega$ is continuous.
Let $t^*\in[0,T]$ and $(t_n)_{n\in\mathbb{N}_{>0}}$ be such that $t_n\rightarrow t^*$ as $n\rightarrow\infty$.
Then with the continuity of $f$ and $\varphi$ on $[0,T]$,
\[\begin{aligned}
\lim_{n\rightarrow \infty} f(\xi(t_n),t_n) & = \lim_{n\rightarrow \infty} \left( \frac{1}{|\varphi(t_n)|_{L^1(\Omega)}} \langle f(t_n),\varphi(t_n) \rangle_{L^2(\Omega)} \right) \\ 
& = \frac{1}{|\varphi(t^*)|_{L^1(\Omega)}} \langle f(t^*),\varphi(t^*) \rangle_{L^2(\Omega)} = f(\xi(t^*),t^*). 
\end{aligned}\]
Therefore, $\xi\in C([0,T];\Omega)$. This finishes the proof.
\end{proof}


\begin{thebibliography}{10}

\bibitem{Babin&Vishik92}
A.~V. Babin and M.~I. Vishik, \emph{Attractors of evolution equations},
  North-Holland, Amsterdam, 1992.

\bibitem{Ball00}
J.~M. Ball, \emph{Continuity properties and global attractors of generalized
  semiflows and the {N}avier-{S}tokes equations}, Nonlinear Science \textbf{7}
  (1997), no.~5, 475--502, Corrected version appears in the book {\em
  Mechanics: From Theory to Computation}, Springer-Verlag, New York 447--474,
  2000.

\bibitem{Ball04}
\bysame, \emph{Global attractors for damped semilinear wave equations},
  Discrete Contin. Dyn. Syst. \textbf{10} (2004), no.~2, 31--52.

\bibitem{Chepyzhov&Vishik02}
Vladimir~V. Chepyzhov and Mark~I. Vishik, \emph{Attractors for equations of
  mathematical physics}, Colloquium Publications - Volume 49, American
  Mathematical Society, Providence, 2002.

\bibitem{Fujita-66}
H.~Fujita, \emph{On the blowing up of solutions to the cauchy problem for $u_t
  = \delta u + u^{1+\alpha}$}, J. Fac. Sci. Univ. Tokyo, Sect. IA, Math.
  \textbf{13} (1966), 109--124.

\bibitem{Hunter&Nachtergaele01}
John~K. Hunter and Bruno Nachtergaele, \emph{Applied analysis}, World
  Scientific Publishing Company, Hakensack, 2001.

\bibitem{Ladyzhenskaya91}
Olga Ladyzhenskaya, \emph{Attractors for semigroups and evolution equations},
  Cambridge University Press, Cambridge, 1991.

\bibitem{Melnik&Valero98}
V.~S. Melnik and J.~Valero, \emph{On attractors of multi-valued semi-flows and
  differential inclusions}, Set-Valued Anal. \textbf{6} (1998), no.~1, 83--111.

\bibitem{Robinson01}
James~C. Robinson, \emph{Infinite-dimensional dynamical systems}, Cambridge
  Texts in Applied Mathematics, Cambridge University Press, Cambridge, 2001.

\bibitem{SR}
P. K. Sahoo and T. Riedel, \emph{Mean value theorems and functional equations}, World Scientific Publishing Co., Inc., River Edge, NJ, 1998.

\bibitem{Segatti06}
Antonio Segatti, \emph{On the hyperbolic relaxation of the {C}ahn-{H}illiard
  equation in 3-d: Approximation and long time behaviour}, Math. Models Methods
  Appl. Sci. \textbf{to appear} (2006).

\bibitem{Temam88}
Roger Temam, \emph{Infinite-dimensional dynamical systems in mechanics and
  physics}, Applied Mathematical Sciences - Volume 68, Springer-Verlag, New
  York, 1988.

\bibitem{Wang13}
Chunpeng Wang, \emph{Asymptotic behavior of solutions to a class of semilinear
  parabolic equations with boundary degeneracy}, Proc. Amer. Math. Soc.
  \textbf{141} (2013), no.~9, 3125--3140.

\bibitem{Zelik04}
Sergey Zelik, \emph{Asymptotic regularity of solutions of singularly perturbed
  damped wave equations with supercritical nonlinearities}, Discrete Contin.
  Dyn. Syst. \textbf{11} (2004), no.~3, 351--392.

\bibitem{Zheng04}
Songmu Zheng, \emph{Nonlinear evolution equations}, Monographs and Surveys in
  Pure and Applied Mathematics - Volume 133, Chapman \& Hall/CRC, Boca Raton,
  2004.
\end{thebibliography}

\providecommand{\bysame}{\leavevmode\hbox to3em{\hrulefill}\thinspace}
\providecommand{\MR}{\relax\ifhmode\unskip\space\fi MR }
\providecommand{\MRhref}[2]{%
  \href{http://www.ams.org/mathscinet-getitem?mr=#1}{#2}
}
\providecommand{\href}[2]{#2}

\end{document}